\documentclass[12pt]{article}
\usepackage{amssymb,amsfonts,amsthm}
\usepackage{amstext}

\usepackage{amsmath}

\usepackage{color}
\usepackage[latin1]{inputenc}
\def\openC{{\rm C\kern-.18cm\vrule width.8pt height 7pt depth-.2pt \kern.18cm}}
\def\openN{{{\rm I}\kern-.16em {\rm N}}}
\def\openR{{{\rm I}\kern-.16em {\rm R}}}
\def\openT{{{\rm T}\kern-.42em {\rm T}}}
\def\openZ{{{\rm Z}\kern-.28em{\rm Z}}}

\newtheorem{thm}{Theorem}[section]
\newtheorem{cor}[thm]{Corollary}
\newtheorem{lem}[thm]{Lemma}

\theoremstyle{definition}

\begin{document}

\title{\textbf{Strictly positive definite kernels on two-point compact homogeneous spaces} \vspace{-2pt}
\author{\sc
V. S. Barbosa\ \,\,\&\,\, V. A. Menegatto}}
\date{}
\maketitle \vspace{-30pt}
\bigskip
\begin{center}
\parbox{13 cm}{{\small We present a necessary and sufficient condition for the strict positive definiteness of a real, continuous, isotropic and positive definite kernel on a two-point compact homogeneous space.\ The characterization adds to others previously obtained by D. Chen at all (2003) in the case in which the space is a sphere of dimension at least 2 and Menegatto at all (2006) in the case in which the space is the unit circle.\
As an application, we use the characterization to improve upon a recent result on the differentiability of positive definite kernels on the spaces. }}
\end{center}
\noindent{\bf 2010 Math. Subj. Class.:} 22F30; 33C50; 33C55; 41A63; 42A82.

\thispagestyle{empty}

%
%-------------------------------------------------------------------------------------------------------------------------------------------------------
%
\section{Introduction}
Let $\mathbb{M}^d$ denote a $d$-dimensional compact two-point homogeneous space.\ As pointed by Wang (\cite{wang}), $\mathbb{M}^d$ belongs to one of the following categories: the unit circle $S^1$, higher dimensional unit spheres $S^d$, $d=2,3\ldots$, the real projective spaces $\mathbb{P}^d(\mathbb{R})$, $d=2,3,\ldots$, the complex projective spaces
$\mathbb{P}^d(\mathbb{C})$, $d=4,6,\ldots$, the quaternionic projective spaces $\mathbb{P}^d(\mathbb{H})$, $d=8,12,\ldots$, and the Cayley projective plane
$\mathbb{P}^{d}(Cay)$, $d=16$.

In this paper, we will deal with real, continuous, isotropic (zonal), and positive definite kernels on $\mathbb{M}^d$.\ The {\em positive definiteness} of a real and symmetric kernel $K$ on $\mathbb{M}^d$ requires that
\begin{equation}\label{pd}
\sum_{\mu,\nu=1}^n c_\mu c_\nu K(x_\mu,x_\nu) \geq 0,
\end{equation}
whenever $n$ is a positive integer, $x_1,x_2, \ldots, x_n$ are distinct points on $\mathbb{M}^d$ and $c_1,c_2, \ldots, c_n$ are real scalars.\  The continuity of $K$ can be defined through the usual Riemannian (geodesic) distance on $\mathbb{M}^d$, here assumed to be normalized so that all geodesics on $\mathbb{M}^d$ have the same length $2\pi$.\ The distance between two points $x,y \in \mathbb{M}^d$ will be written as $|xy|$.\ Since $\mathbb{M}^d$ possesses a group of motions $G_d$ which takes any pair of points $(x,y)$ to $(z,w)$ when $|xy|=|zw|$, {\em isotropy} of a kernel $K$ on $\mathbb{M}^d$ will refer to the property
\begin{equation*}
K(x,y)=K(Ax,Ay), \quad x,y \in \mathbb{M}^d, \quad A \in G_d.
\end{equation*}
An isotropic kernel $K$ on $\mathbb{M}^d$ can be written in the form
\begin{equation*}
K(x,y)=K_r^d(\cos{(|xy|/2)}), \quad x,y \in \mathbb{M}^d,
\end{equation*}
for some function $K_r^d: [-1,1] \to \mathbb{R}$, here called the {\em isotropic part} of $K$.

According to \cite{bochner,gangolli}, a real, continuous and isotropic kernel $K$ on $\mathbb{M}^d$ is positive definite if and only if the isotropic part $K_r^d$ of $K$ has a series representation in the form
\begin{equation}\label{PDM}
K_r^d(t)=\sum_{k=0}^{\infty}a_k^{\alpha,\beta} P_k^{\alpha,\beta}(t), \quad t \in [-1,1],
\end{equation}
in which $a_k^{\alpha,\beta} \in [0,\infty)$, $k\in \mathbb{Z}_+$ and $\sum_{k=0}^{\infty}a_k^{\alpha,\beta} P_k^{\alpha,\beta}(1) <\infty$.\ The first upper exponent $\alpha$ depends only on the dimension $d$ and is given by $\alpha:=(d-2)/2$, whereas $\beta$ can take the values $(d-2)/2, -1/2, 0, 1, 3$, depending on the respective category $\mathbb{M}^d$ belongs to, among the five we have mentioned at the beginning of the paper.\ The symbol $P_k^{\alpha,\beta}$ stands for usual the Jacobi polynomial of degree $k$ associated with the pair $(\alpha,\beta)$ as defined in \cite{szego}.

The intended target in the present paper is to characterize the real, continuous, isotropic and strictly positive definite kernels on $\mathbb{M}^d$.\ A positive definite kernel $K$ on $\mathbb{M}^d$ is {\it strictly positive definite} if the inequality \eqref{pd} is strict for $n\geq 1$, distinct points $x_1,x_2,\hdots,x_n\in \mathbb{M}^d$ and scalars $c_1,c_2,\hdots,c_n$ not simultaneously zero.\ A characterization in the cases in which $\mathbb{M}^d$ is a sphere $S^d$ was previously obtained as we now update.\ In the case of $S^1$, the additional condition for the strict positive definiteness of the kernel is that the set
\begin{equation*}
\left\{k \in \mathbb{Z}: a_{|k|}^{-1/2,-1/2}>0\right\}
\end{equation*}
have a nonempty intersection with every full arithmetic progression in $\mathbb{Z}$.\ The actual arguments that lead to this characterization are in Theorems 2.2 and 2.9 in \cite{mene}.\ An independent and direct proof will be sketched in an appendix at the end of the present paper, for the convenience of the reader.\  As for $S^d$, $d\geq 2$, the additional condition is much simpler (\cite{chen}): the set
\begin{equation*}
\left\{k \in \mathbb{Z}_+: a_k^{(d-2)/2,(d-2)/2}>0\right\}
\end{equation*}
needs to contain infinitely many even integers and infinitely many odd integers.\ The main contribution in this paper is the following complement to the above information.

\begin{thm} \label{main}  Let $K$ be a real, continuous, isotropic and positive definite kernel on $\mathbb{M}^d$, $d\geq 2$.\ If $\mathbb{M}^d\neq S^d$, then $K$ is strictly positive definite if and only if the set $\{k \in \mathbb{Z}_+: a_k^{(d-2)/2,\beta}>0\}$ is infinite.
\end{thm}

As mentioned before, in the case in which $\mathbb{M}^d$ is a sphere of dimension at least 2, this theorem was originally proved in \cite{chen}.\ However, the proof presented in that reference comprehended specific arguments with coordinate systems on the sphere.\ An additional contribution of the present paper resides in the fact that the proof of Theorem \ref{main} implies a much simpler proof for the characterization of strict positive definiteness on $S^d$, $d\geq 2$, obtained in \cite{chen}.

An outline of the paper is as follows.\ Section 2 contains the basic material to be used in the proof of the main theorem: a new apparel for  the concept of strict positive definiteness, a key limit property of Jacobi polynomials and isometric embeddings of the spaces $\mathbb{M}^d$ among themselves.\ Section 3 contains the actual proof of Theorem \ref{main} while an application of the theorem to differentiability of positive definite kernels on $\mathbb{M}^d$ is the content of Section 4.\ It complements the main theorem proved in \cite{barbosa}.\ Finally, an alternative proof for the characterization of the strict positive definiteness
of a real, continuous, isotropic and positive definite kernel on $S^1$ appears in the Appendix section.

\section{Basic notation and technical results}

In this section and the others to come, we will assume the upper indices $\alpha$ and $\beta$ belong to the scope considered in the previous section.\ Here, we will include the technical results to be used in
the proof of the main result of the paper.

The spectrum of the Laplace-Beltrami operator $\Delta_d$ on $\mathbb{M}^d$ is discrete, real and non-positive, so that its elements can be arranged in decreasing order, say,
$0=\lambda_0>\lambda_1>\lambda_2>\hdots$.\ If $\mathcal{H}^d_k$ is the eigenspace of $\Delta_d$ corresponding to the eigenvalue $\lambda_k$, it is well known that
the spaces $\mathcal{H}^d_k$ are mutually orthogonal in $L^2(\mathbb{M}^d,\sigma_d)$, in which $\sigma_d$ is the normalized Riemannian measure on $\mathbb{M}^d$.\
We will write $\{S^d_{k,1},S^d_{k,2},\hdots,S^d_{k,\delta(k,d)}\}$ to denote an orthonormal basis of $\mathcal{H}^d_k$ with respect to the inner product of the space above while $\delta(k,d)$ is its dimension.\
A result of Giné (\cite{gine,koo}) justifies the addition formula
\begin{equation*}
\sum_{j=1}^{\delta(k,d)} S_{k,j}^d(x)\overline{S_{k,j}^d(y)} = c_k^{\alpha,\beta} P_k^{\alpha,\beta}\left(\cos{(|xy|/2)}\right),\quad x,y\in \mathbb{M}^d,
\end{equation*}
where
\begin{equation*}
c_{k}^{\alpha,\beta} := \frac{\Gamma(\beta+1)(2k+\alpha+\beta+1)\Gamma(k+\alpha+\beta+1)}{\Gamma(\alpha+\beta+2)\Gamma(k+\beta+1)}.
\end{equation*}
More information on this formula and on the general harmonic analysis on $\mathbb{M}^d$ can be found in the references \cite{bordin,brown,kush1, platonov}.

The following lemma has extreme importance in the analysis of the strict positive definiteness pertaining to this paper.

\begin{lem}\label{matrix} Let $K$ be a nonzero, real, continuous, isotropic and positive definite kernel $K$ on $\mathbb{M}^d$, $x_1, x_2, \ldots, x_n$ distinct points on $\mathbb{M}^d$, and $c_1, c_2, \ldots, c_n$ real scalars.\ Consider the representation (\ref{PDM}) for the isotropic part $K_r^d$ of $K$.\ The following statements are equivalent:
\begin{enumerate}
\item[$(i)$] $\sum_{\mu,\nu=1}^n c_\mu c_\nu K(x_\mu,x_\nu)=0$;
\item[$(ii)$] The equality
\begin{equation*}
\sum_{\mu=1}^n c_\mu P^{\alpha,\beta}_k(\cos{(|x_\mu x|/2)})=0
\end{equation*}
holds for all $x\in \mathbb{M}^d$ and all $k$ in the set $\{k: a_k^{\alpha,\beta}>0\}$.
\end{enumerate}
\end{lem}
\begin{proof} Let us write $c^tAc$ to denote the quadratic form in $(i)$.\ Introducing \eqref{PDM} and the addition formula in the quadratic form and arranging leads to
\begin{align*}
c^tAc & =  \sum_{\mu,\nu=1}^n c_\mu c_\nu \sum_{k=0}^{\infty}a_k^{\alpha,\beta} P_k^{\alpha,\beta}(\cos{(|x_\mu x_\nu|/2)})\\
 & =  \sum_{\mu,\nu=1}^n c_\mu c_\nu \sum_{k=0}^{\infty}\frac{a_k^{\alpha,\beta}}{c_k^{\alpha,\beta}} \sum_{j=1}^{\delta(k,d)} S_{k,j}^d(x_\mu)\overline{S_{k,j}^d(x_\nu)}\\
& =  \sum_{k=0}^\infty \frac{a_k^{\alpha,\beta}}{c_{k}^{\alpha,\beta}}\sum_{j=1}^{\delta(k,d)} {\left| \sum_{\mu=1}^n c_\mu S_{k,j}^d(x_\mu)\right|}^2.
\end{align*}
Thus, $\sum_{\mu,\nu=1}^n c_\mu c_\nu K(x_\mu,x_\nu)=0$ if and only if
\begin{equation*}
\sum_{\mu=1}^n c_\mu S_{k,j}^d(x_\mu) = 0,\quad j\in \{1,2,\hdots, \delta(k,d)\},\quad k \in\{k: a_k^{\alpha,\beta}>0\}.
\end{equation*}
Multiplying by $\overline{S_{k,j}^d(x)}$ and adding up on $l$, we obtain that
\begin{equation*}
\sum_{\mu=1}^n c_\mu \sum_{j=1}^{\delta(k,d)}S_{k,j}^d(x_\mu)\overline{S_{k,j}^d(x)} = 0,\quad x\in \mathbb{M}^d, \quad k \in \{k:  a_k^{\alpha,\beta}>0\}.
\end{equation*}
Another application of the addition formula leads to the statement in $(ii)$.\\
Conversely, if $(ii)$ holds, then
\begin{equation*}
\sum_{\mu=1}^n c_\mu \sum_{j=1}^{\delta(k,d)} S_{k,j}^d(x_\mu)\overline{S_{k,j}^d(x)}=0,\quad x\in \mathbb{M}^d, \quad k \in\{k: a_k^{\alpha,\beta}>0\},
\end{equation*}
that is,
\begin{equation*}
\sum_{j=1}^{\delta(k,d)}\left[ \sum_{\mu=1}^n c_\mu \overline{S_{k,j}^d(x_\mu)}\right] S_{k,j}^d(x) =0,\quad x\in \mathbb{M}^d,\quad k \in\{k: a_k^{\alpha,\beta}>0\}.
\end{equation*}
Since $\{S^d_{k,1},S^d_{k,2},\hdots,S^d_{k,\delta(k,d)}\}$ is a basis of $\mathcal{H}_k^d$,
\begin{equation*}
\sum_{\mu=1}^n c_\mu S_{k,j}^d(x_\mu) = 0, \quad j\in \{1,2,\hdots, \delta(k,d)\},\quad k \in\{k:  a_k^{\alpha,\beta}>0\}.
\end{equation*}
Now, the computations made in the first half of the proof imply that the equality $c^tAc=0$ holds.
\end{proof}

In some points of the paper, normalized Jacobi polynomials will be more suitable.\ We will write
\begin{equation*}
R_k^{\alpha,\beta}:=\frac{P_k^{\alpha,\beta}}{P_k^{\alpha,\beta}(1)}, \quad k\in \mathbb{Z}_+,
\end{equation*}
in which
\begin{equation*}
P_k^{\alpha,\beta}(1)=\left( \begin{array}{c} k+\alpha\\ k \end{array}\right):= \frac{\Gamma(k+\alpha+1)}{k!\Gamma(\alpha+1)}, \quad k \in \mathbb{Z}_+,\quad \alpha,\beta >-1.
\end{equation*}
As usual, $\Gamma$ denotes de Gamma function.

Below, we collect a few properties of the Jacobi polynomials.

\begin{lem}\label{limit} The Jacobi polynomials have the following properties:\\
$(i)$ $P_k^{\alpha,\beta}(-t)=(-1)^k P_k^{\beta,\alpha}(t)$, $t \in [-1,1]$;\\
$(ii)$ $\lim_{k\to \infty} R_{k}^{\alpha,\beta}(t)=0$, $t \in (-1,1)$;\\
$(iii)$ If $\alpha>\beta$, then $\lim_{k\to \infty} P_k^{\beta,\alpha}(1)[P_k^{\alpha,\beta}(1)]^{-1}=0$.
\end{lem}
\begin{proof} Property $(i)$ is a classical result in the theory of orthogonal polynomials (\cite[p.59]{szego}).\ Property $(ii)$ follows from a formula derived in \cite[p.196]{szego}.\ However, for the values of $\alpha$ and $\beta$ pertinent to this paper, except for the case in which $\alpha=0=\beta+1/2$, it can also be accomplished via the recurrence formula (\cite[p.71]{szego})
\begin{equation*}
(1-t)R_k^{\alpha,\beta}(t) = \frac{2\alpha}{2k+\alpha+\beta+1}\left[R_k^{\alpha-1,\beta}(t)- R_{k+1}^{\alpha-1,\beta}(t)\right], \quad k\in \mathbb{Z}_+,\quad t\in (-1,1),
\end{equation*}
by taking a limit as $k \to \infty$.\ The limit formula in $(iii)$ follows from the identity
\begin{equation*}
\frac{P_k^{\beta,\alpha}(1)}{P_k^{\alpha,\beta}(1)} = \frac{\Gamma(\alpha+1)}{\Gamma(\beta+1)}\frac{\Gamma(k+\beta+1)}{\Gamma(k+\alpha+1)}, \quad k\in \mathbb{Z}_+,
\end{equation*}
and an application of the following well known limit formula
\begin{equation*}
\lim_{n\to \infty} \frac{\Gamma(n+x)}{\Gamma(n) n^{x}}=1, \quad x \in \mathbb{R}
\end{equation*}
for the Gamma function.
\end{proof}

The last result in this section refers to isometric embeddings in $\mathbb{M}^d$.\ If $(M_1,d_1)$ and $(M_2,d_2)$ are metric spaces, an {\em isometric embedding} from $M_1$ into $M_2$ is a function
$\phi : M_1 \to M_2$ for which
\begin{equation*}
d_2(\phi(x),\phi(y))=d_1(x,y),\quad x,y \in M_1.
\end{equation*}

If we write $M_1 \hookrightarrow M_2$ to indicate the existence of an isometric embedding from $M_1$ to $M_2$, the following result holds (see \cite[p.66]{askey} and references therein).\ It is worth to mention that it takes into account the distance normalization we have adopted for the
 metric spaces $\mathbb{M}^d$.

\begin{lem}\label{embed} There exists a chain of isometric embeddings as follows
\begin{equation*}
S^1\hookrightarrow \mathbb{P}^{2}(\mathbb{R}) \hookrightarrow \mathbb{P}^{d}(\mathbb{R})\hookrightarrow \mathbb{P}^{2d}(\mathbb{C})\hookrightarrow \mathbb{P}^{4d}(\mathbb{H})\hookrightarrow \mathbb{P}^{8d}(Cay), \quad d=2,3,\hdots.
\end{equation*}
\end{lem}

An obvious consequence of the previous lemma is that $S^1$ can be isometrically embedded in all the $\mathbb{M}^d$, $d\geq 2$.

\section{Strict positive definiteness}

The results in this section will converge to a characterization for the real, continuous, isotropic and strictly positive definite kernels on $\mathbb{M}^d$.\ Our proofs restricted to the cases in which the The proofs also include the case in which $\mathbb{M}^d$ is a sphere of dimension at least 2 are somehow more elegant than that presented in \cite{chen}.\ However, the proofs in this section do not apply to case in which the space is a circle.

We begin with a necessary condition.\ Here, it is convenient to prove the theorem for all the homogeneous spaces considered in the paper.

\begin{thm} \label{nec} Let $K$ be a nonzero, real, continuous, isotropic and positive definite kernel on $\mathbb{M}^d$.\ In order that it be strictly positive definite it is necessary that in the representation (\ref{PDM}) for $K_r^d$, $a_k^{\alpha,\beta}>0$ for infinitely many integers $k$.\ If $\alpha=\beta$, then it is also necessary that $a_k^{\alpha,\beta}>0$ for infinitely many even and infinitely many odd $k$.
\end{thm}
\begin{proof} In the first half of the proof we will show that if $\{k: a_{k}^{\alpha,\beta}>0\}$ is finite, then $K$ is not strictly positive definite on $\mathbb{M}^d$.\ Let $\phi :S^1 \to \mathbb{M}^d$ be an isometric embedding, as guaranteed by Lemma \ref{embed}.\ This embedding allows the selection of $n$ distinct points $x_1, x_2, \ldots, x_{n}$ in $\mathbb{M}^d$ so that
\begin{equation*}
|x_\mu x_\nu|=|\phi^{-1}(x_\mu)\, \phi^{-1}(x_\nu)|, \quad \mu,\nu=1,2,\ldots,n.
\end{equation*}
We now look at the $n \times n$ matrix with $\mu\nu$-entry
\begin{equation*}
K(x_\mu,x_\nu)=K_r^d(\cos( |\phi^{-1}(x_\mu)\, \phi^{-1}(x_\nu)|/2))=\sum_{k=0}^N a_{k}^{\alpha,\beta}P_k^{\alpha,\beta}(\phi^{-1}(x_\mu)\cdot  \phi^{-1}(x_\nu)).
\end{equation*}
in which $N:=\max\{k: a_{k}^{\alpha,\beta}>0\}$ and $\cdot $ is the usual inner product of $\mathbb{R}^2$.\ Since $\{x_1,x_2, \ldots, x_n\}$ is a subset of $\mathbb{R}^2$, the powered Gram matrix with entries $(\phi^{-1}(x_\mu)\cdot  \phi^{-1}(x_\nu))^k$ has rank
at most $2^k$.\ In particular the matrix $[K(x_\mu,x_\nu)]$ has rank at most $\sum_{k=0}^N 2^k= 2^{N+1}-1,$ a number that does not depend upon $n$.\ In particular, if $n\geq 2^{N+1}$, the matrix $[K(x_\mu,x_\nu)]$ cannot be of full rank.\ This takes care of the first assertion.\ Next, we assume $\alpha=\beta$.\ We will show that if $\{k: a_{2k}^{\alpha,\alpha}>0\}$ is finite, then $K$ is not strictly positive definite on $\mathbb{M}^d$.\ The other half of the proof, under the assumption that $\{k: a_{2k+1}^{\alpha,\alpha}>0\}$ is finite, is similar and will not be sketched.\ If $\{k: a_{2k}^{\alpha,\alpha}>0\}$ is finite, then we can write
\begin{equation*}
K_r^d(t)=\sum_{k=0}^{2N'} a_{k}^{\alpha,\alpha} P_{k}^{\alpha,\alpha}(t) + \sum_{k\geq 2N'+1} a_{k}^{\alpha,\alpha} P_{k}^{\alpha,\alpha}(t), \quad t\in [-1,1],
\end{equation*}
in which $N'=\max\{k: a_{2k}^{\alpha,\alpha}>0\}$.\ Now, we select $2n$ distinct points $y_1, y_2, \ldots, y_{2n}$ in $\mathbb{M}^d$ so that
\begin{equation*}
\phi^{-1}(y_\mu)=-\phi^{-1}(y_{n+\mu}),\quad \mu=1,2,\hdots,n,
\end{equation*}
\begin{equation*}
|y_\mu y_\nu|=|\phi^{-1}(y_\mu)\, \phi^{-1}(y_\nu)|, \quad \mu,\nu=1,2,\ldots,2n,
\end{equation*}
and define $2n \times 2n$ matrices $A$ and $B$ with entries given by
\begin{equation*}
B_{\mu\nu} = \sum_{k=0}^{2N'} a_{k}^{\alpha,\alpha} P_{k}^{\alpha,\alpha}(\cos{(|x_\mu x_\nu|/2})) \quad A_{\mu\nu} = K_r^d(\cos{(|x_\mu x_\nu|/2)}).
\end{equation*}
If for $\mu \in \{1,2,\ldots,n\}$, $c_\mu$ is the vector having its $\mu$-th and $(n + \mu)$-th components equal to $1$ and all the others equal to $0$, it is promptly seen that $c_\mu$ belongs to the kernel of $A-B$.\ In particular, the rank of $A-B$ is at most $n$.\ Taking into account the first half of the proof, it is now clear that the rank of $A=[K(x_\mu, x_\nu)]$ does not exceed $2^{2N'+1}-1+n$.\ Therefore, if $n>2^{2N'+1}-1$, the matrix $A$ cannot be of full rank.
\end{proof}

Next, we will prepare the terrain for the proof of the sufficiency part of the conditions presented in the previous theorem.\ For a fixed $x \in \mathbb{M}^d$, the {\em antipodal manifold} of $x$ is the set
\begin{equation*}
\Gamma_x:=\{y \in \mathbb{M}^d: |xy|=2\pi\}.
\end{equation*}
The following lemma was originally proved by E. Cartan (\cite{cartan}) and T. Nagano (\cite{nagano}).\ But, it is also quoted and re-obtained in \cite{helgason,tirao}.\ It describes what the antipodal manifold of a point in $\mathbb{M}^d$ is.

\begin{lem} Let $x$ be a fixed point in $\mathbb{M}^d$.\ The antipodal manifold $\Gamma_x$ of $x$ is a point if $\mathbb{M}^d=S^d$ and is isometrically isomorphic to $\mathbb{P}^{d-1}(\mathbb{R})$, $\mathbb{P}^{d-2}(\mathbb{C})$, $\mathbb{P}^{d-4}(\mathbb{H})$, and $S^8$ in the cases $\mathbb{M}^d$ is respectively, $\mathbb{P}^{d}(\mathbb{R})$, $\mathbb{P}^{d}(\mathbb{C})$, $\mathbb{P}^{d}(\mathbb{H})$, and $\mathbb{P}^{d}(Cay)$.
\end{lem}

In particular, the previous lemma reveals that if $\mathbb{M}^d$ is not a sphere, then for a fixed point $x$ in $\mathbb{M}^d$, there are infinitely many $y$ in $\mathbb{M}^d$ for which $\cos{(|x y|/2)} =-1$.\ This will have significance in the arguments in the proof of the next theorem.

\begin{thm} Let $K$ be a real, continuous, isotropic and positive definite kernel on $\mathbb{M}^d$, $d\geq 2$.\ In order that it be strictly positive definite it is sufficient that in the representation (\ref{PDM}) for $K_r^d$, $a_k^{\alpha,\beta}>0$ for infinitely many odd and infinitely many even integers $k$.\ If $\alpha >\beta$, the condition can be weakened to $a_k^{\alpha,\beta}>0$ for infinitely many integers $k$.
\end{thm}
\begin{proof} Assume $a_k^{\alpha,\beta}>0$ for infinitely many even and infinitely many odd integers $k$, let $n$ be a positive integer and $x_1,x_2,\hdots,x_n$ distinct points in $\mathbb{M}^d$.\ As before, write $A$ to denote the $n\times n$ matrix with entries $A_{\mu\nu}:=K_r^d(\cos{|x_\mu x_\nu|/2})$.\ We intend to show that the equality $\sum_{\mu,\nu=1}^n c_\mu c_\nu A_{\mu\nu}=0$ implies $c_\mu=0$ for all $\mu\in\{1,2,\hdots,n\}$.\ Due to Lemma \ref{matrix}, that corresponds to showing that the only solution of the system
\begin{equation*}
\sum_{\mu=1}^n c_\mu P^{\alpha,\beta}_k(\cos{(|x_\mu x|/2)})=0, \quad x\in \mathbb{M}^d,\quad k\in \{k: a_k^{\alpha,\beta}>0\},
\end{equation*}
is $c_1=c_2=\cdots=c_n=0$.\ In order to achieve that, we will fix an arbitrary coordinate index $\gamma \in \{1,2,\hdots,n\}$ and will conclude that $c_\gamma=0$ via a specific choice for the point $x\in \mathbb{M}^d$ in the system above.\ There are two cases to be considered:\\
\underline{Case 1:} $\cos{(|x_\mu x_\gamma|/2)} \neq -1$, $\mu \neq \gamma$.\\
Choosing $x=x_\gamma$, the system reduces itself to
\begin{equation*}
c_\gamma P^{\alpha,\beta}_k(1)+\sum_{\mu \neq \gamma} c_\mu P^{\alpha,\beta}_k(\cos{(|x_\mu x_\gamma|/2)})=0,\quad k \in \{k: a_k^{\alpha,\beta}>0\}.
\end{equation*}
Since $a_k^{\alpha,\beta}>0$ for infinitely many integers $k$, we can select a sequence $\{k_r\}_{r\in \mathbb{Z}_+}$ of positive integers for which $a_{k_r}^{\alpha,\beta}>0$, $r\in \mathbb{Z}_+$ and $\lim_{r\to \infty}k_r=\infty$.\ Introducing this sequence in the previous system, we are left with
\begin{equation*}
c_\gamma + \sum_{\mu \neq \gamma} c_\mu R^{\alpha,\beta}_{k_r}(\cos{(|x_\mu x_\gamma|/2)})=0,\quad r\in \mathbb{Z}_+.
\end{equation*}
Since $|\cos{(|x_\mu x_\gamma|/2)}| \neq \pm 1$, $\mu \neq \gamma$, Lemma \ref{limit}-$(ii)$ implies that
\begin{equation*}
0=c_\gamma + \lim_{r\to \infty} \sum_{\mu \neq \gamma} c_\mu R^{\alpha,\beta}_{k_r}(\cos{(|x_\mu x_\gamma|/2)}) = c_\gamma.
\end{equation*}
\underline{Case 2:} $\cos{(|x_\mu x_\gamma|/2)}=- 1$, for at least one $\mu \neq \gamma$.\\
In this case we need to consider the antipodal manifold $\Gamma$ of $x_\gamma$.\ The same choice $x=x_\gamma$ supplies the following sub-system
$$
c_\gamma + (-1)^k \frac{P_k^{\beta,\alpha}(1)}{P_k^{\alpha,\beta}(1)}\sum_{x_\mu \in \Gamma}c_\mu +
         \sum_{x_\mu \not \in \Gamma\cup\{x_\gamma\}}c_\mu R^{\alpha,\beta}_k(\cos{(|x_\mu x_\gamma|/2)})=0,\quad k \in \{k:a_k^{\alpha,\beta}>0\}.$$
We now break the proof into two subcases.\\
\underline{Subcase $\alpha > \beta$:} Here we can select a sequence $\{k_r\}_{r \in \mathbb{Z}_+} \subset \mathbb{Z}_+$ so that $a_{k_r}^{\alpha,\beta}>0$, $r\in \mathbb{Z}_+$, and $\lim_{r\to\infty}k_r=\infty$.\
Introducing it the main equation and letting $r \to \infty$ we reach
$$
c_\gamma + \left[\lim_{r \to \infty}(-1)^{k_r} \frac{P_{k_r}^{\beta,\alpha}(1)}{P_{k_r}^{\alpha,\beta}(1)}\right]\sum_{x_\mu \in \Gamma}c_\mu +
         \sum_{x_\mu \not \in \Gamma\cup\{x_\gamma\}}c_\mu \left[\lim_{r \to \infty} R^{\alpha,\beta}_{k_r}(\cos{(|x_\mu x_\gamma|/2)})\right]=0.
         $$
The first limit above is zero due to Lemma \ref{limit}-$(iii)$ while the second one is zero due to Lemma \ref{limit}-$(ii)$.\ Thus, $c_\gamma=0$.\\
\underline{Subcase $\alpha=\beta$:} Here, we select two sequences $\{k_r\}_{r\in\mathbb{Z}_+} \subset 2\mathbb{Z}_+$ and $\{k_s\}_{s\in\mathbb{Z}_+} \subset \mathbb{Z}_+\setminus2\mathbb{Z}_+$
so that $a_{k_r}^{\alpha,\beta}a_{k_s}^{\alpha,\beta}>0$, $r,s\in \mathbb{Z}_+$, and $\lim_{r\to\infty}k_r=\lim_{s \to \infty} k_s=\infty$.\ Introducing them in the main equation and using Lemma \ref{limit}-$(ii)$ once again, the outcome is
\begin{equation*}
c_\gamma + \sum_{x_\mu \in \Gamma}c_\mu=c_\gamma - \sum_{x_\mu \in \Gamma}c_\mu=0.
\end{equation*}
Once again, $c_\gamma=0$.
\end{proof}

The formula
\begin{equation}\label{geg}
P_{2k}^{(d-1)/2}(t)=\frac{P_{2k}^{(d-1)/2}(1)}{P_k^{(d-2)/2,-1/2}(1)}P_k^{(d-2)/2,-1/2}(2t^2-1),\quad t \in [-1,1], \quad k \in \mathbb{Z}_+,
\end{equation}
allows an alternative description for a real, continuous, isotropic and positive definite kernel $K$ on $P^{d}(\mathbb{R})$ via the Gegenbauer polynomials $P_{2k}^{(d-1)/2}$ (\cite[p.59]{szego}).\ In  fact, by \eqref{geg} we have
\begin{equation*}
P_{2k}^{(d-1)/2}(\cos(|xy|/4))=\frac{P_{2k}^{(d-1)/2}(1)}{P_k^{(d-2)/2,-1/2}(1)}P_k^{(d-2)/2,-1/2}(\cos(|xy|/2)),\quad t \in [-1,1], \quad k \in \mathbb{Z}_+,
\end{equation*}
that is,
\begin{equation*}
K_r^d(x,y)=\sum_{k=0}^\infty a_{2k}P_{2k}^{(d-1)/2}(\cos |xy|/4), \quad t\in [-1,1],
\end{equation*}
in which $a_{2k}\geq 0$ is a positive multiple of $a_k^{(d-2)/2,-1/2}$ and $\sum_{k=0}^\infty a_{2k}P_{2k}^{(d-1)/2}(1)<\infty$.\ Hence, the following alternative description holds.

\begin{cor} ($d\geq 2$) Let $K$ be a real, continuous, isotropic and positive definite kernel on $P^d(\mathbb{R})$ written as above.\ It is strictly positive definite on $P^d(\mathbb{R})$ if and only if $a_{2k}>0$ for infinitely many integers.
\end{cor}

\section{An application on differentiability}

It is known that the isotropic part of a real, continuous, isotropic, and positive definite kernel on $\mathbb{M}^d$ is differentiable up to a certain order (that depends upon $d$) in $(-1,1)$ (\cite{barbosa}).\ As a matter of fact, the following description was obtained in \cite{barbosa}, as a generalization of another one proved in \cite{ziegel}.

\begin{thm}\label{diff} If $K$ is a real, continuous, isotropic and positive definite kernel on $\mathbb{M}^d$, $d\geq 3$, then the isotropic part $K_r^d$ of $K$ is continuously differentiable on $(-1,1)$.\ The derivative $(K_r^d)'$
of $K_r^d$ in $(-1,1)$ satisfies a relation of the form
\begin{equation*}
(1-t^2)(K_r^d)'(t) = f_1(t) - f_2(t),\quad t\in(-1,1),
\end{equation*}
in which $f_1$ and $f_2$ are the isotropic parts of two continuous, positive definite and zonal kernels on some compact two point homogeneous space $\mathbb{M}$ which is isometrically embedded in $\mathbb{M}^d$.\ The specifics on each case are as follows:\\
$(i)$ $\mathbb{M}^d=S^d$: $d\geq 3$ and $\mathbb{M}=S^{d-2}$;\\
$(ii)$ $\mathbb{M}^d=\mathbb{P}^d(\mathbb{R})$: $d\geq 3$ and $\mathbb{M}=\mathbb{P}^{d-2}(\mathbb{R})$;\\
$(iii)$ $\mathbb{M}^d=\mathbb{P}^d(\mathbb{C})$: $d\geq 4$ and $\mathbb{M}=\mathbb{P}^{d-2}(\mathbb{C})$;\\
$(iv)$ $\mathbb{M}^d=\mathbb{P}^d(\mathbb{H})$: $d\geq 8$, $\mathbb{M}=\mathbb{P}^{d/2-2}(\mathbb{C})$, when $d \in 8\mathbb{Z}_++8$ and $\mathbb{M}=\mathbb{P}^{d/2}(\mathbb{C})$, when $d\in 8\mathbb{Z}_+ +12$;\\
$(v)$ $\mathbb{M}^d=\mathbb{P}^{16}(Cay)$: $\mathbb{M}=S^2$.
\end{thm}

Below, we will need some additional information provided by the proof of Theorem \ref{diff} in \cite{barbosa}.\ But before that, we need to mention the following technical result.

\begin{lem} Let $K$ be a real, continuous, isotropic and positive definite kernel on $\mathbb{M}^d$.\ Assume there exists an isometric embedding from a space $M$ from the list introduced at the beginning of the paper
into $\mathbb{M}^d$.\ Then $K$ is continuous, isotropic and positive definite on $M$.\ In addition, if $K$ is strictly positive definite on $\mathbb{M}^d$, then it is so on $M$.
\end{lem}
\begin{proof} It is quite standard and it will be omitted.
\end{proof}

Based upon the previous lemma, here we will strengthen the notation of the coefficient $a_k^{\alpha,\beta}$ in the expansion (\ref{PDM}) of the isotropic part $K_r^d$ of a real, continuous, isotropic and positive definite kernel $K$ on $\mathbb{M}^d$, by writing $a_k^{\alpha,\beta}(K_r^d)$ instead.\ If $K$ is positive definite on another space $M$, as in the setting of the previous lemma, we may write $a_k^{\alpha',\beta'}(K_r^d)$ with the actual values of $\alpha' <\alpha$ and $\beta'$ that are attached to the space $M$.

\begin{thm}
Under the setting adopted in Theorem \ref{diff}, the functions $f_1$ and $f_2$ have closed forms as follows.\ In cases $(ii)$ and $(iii)$,
\begin{equation*}
f_1=\sum_{n=0}^{\infty} b_n^{\alpha-1,\beta}R_n^{\alpha-1,\beta}
\end{equation*}
and
\begin{equation*}
f_2=\sum_{n=2}^{\infty}\left(b_n^{\alpha,\beta} + b_{n-1}^{\alpha,\beta}\right) R_n^{\alpha-1,\beta},
\end{equation*}
in which the $b_n^{\alpha-1,\beta}$ and the $b_{n}^{\alpha,\beta}$ are positive multiples of $a_n^{\alpha-1,\beta}(K_r^d)$ and $a_n^{\alpha,\beta}(K_r^d)$ respectively.\ In cases $(iii)$ and $(iv)$, %$f_1$ is as before and
\begin{equation*}
f_1=\sum_{n=0}^{\infty} b_n^{\alpha-1,\beta-1}R_n^{\alpha-1,\beta-1}
\end{equation*}
and
\begin{equation*}
f_2=\sum_{n=2}^{\infty}b_{n-1}^{\alpha,\beta} R_n^{\alpha-1,\beta-1}
\end{equation*}
in which the $b_n^{\alpha-1,\beta-1}$ and $b_{n}^{\alpha,\beta}$ are positive multiples of $a_n^{\alpha-1,\beta-1}(K_r^d)$ and $a_{n-1}^{\alpha,\beta}(K_r^d)$ respectively.
\end{thm}

The main theorem in this section is now within reach.\ It demands the following additional isometric embeddings: $\mathbb{P}^{d}(\mathbb{C}) \hookrightarrow \mathbb{P}^{d+2}(\mathbb{C})$, $d=4,6,\ldots$ and $\mathbb{P}^{d}(\mathbb{H}) \hookrightarrow \mathbb{P}^{d+4}(\mathbb{H})$, $d=8,12,\ldots$ (\cite[p.66]{askey} and references therein).

\begin{thm}\label{diffSPD} Under the setting adopted in Theorem \ref{diff}, if the kernel $K$ is strictly positive definite on $\mathbb{M}^d$, then the functions $f_1$ and $f_2$ are actually the isotropic parts of continuous, isotropic, and strictly positive definite kernels on $\mathbb{M}$.
\end{thm}
\begin{proof} We will stress a proof of the theorem in the case $(ii)$ only.\ In case $(i)$, it can be adapted from results in \cite{ziegel} along with the characterization for strict positive definiteness on spheres mentioned at the introduction and proved in \cite{chen}.\ In case $(iii)$, the arguments are similar to those in case $(ii)$.\ If $K$ is strictly positive definite on $\mathbb{P}^d(\mathbb{R})$, it is so on $\mathbb{P}^{d-2}(\mathbb{R})$.\ In particular, $a_n^{\alpha-1,-1/2}(K_r^d)>0$ for infinitely many integers, and the same is true for the $a_n^{\alpha,-1/2}(K_r^d)$.\ Recalling the definitions for the coefficients in the expansions defining $f_1$ and $f_2$ in the previous proposition, we conclude that $b_n^{\alpha-1,-1/2}>0$ for infinitely integers and the same is true for the coefficients $b_n^{\alpha,-1/2} + b_{n-1}^{\alpha,-1/2}$.\ Thus, due to Theorem \ref{main}, $f_1$ and $f_2$ are the isotropic parts of strictly positive definite kernels on $\mathbb{P}^{d-1}(\mathbb{R})$.
\end{proof}

\section{Appendix}

In this section, we include an independent proof for Theorem \ref{main} in the case in which $\mathbb{M}^d =S^1$.\ This case is different from the others in the sense that the additional condition for strict positive definiteness on the coefficients in the expansion of the isotropic part of the kernel has a different structure.\ We emphasize that the characterization to be deduced here was originally obtained in \cite{mene}, but via a complexification approach.

Lemma \ref{matrix} takes the following form (see \cite{xucheney} and references quoted there).

\begin{lem}\label{matrix1} Let $K$ be the isotropic part of a nonzero, continuous, isotropic and positive definite kernel on $S^1$.\ It is strictly positive definite if and only if there exists no non-zero function $f: \mathbb{Z}_+ \to \mathbb{C}$ of the form
\begin{equation*}
f(k)=\sum_{\mu=1}^n c_\mu e^{i\theta_\mu k}, \quad \{\theta_1, \theta_1, \ldots, \theta_n\} \subset [0,2\pi),\quad \{c_1,c_2, \ldots, c_n\} \subset \mathbb{R},
\end{equation*}
 that vanishes on $\{k : a_k^{-1/2,-1/2} >0\}$.
\end{lem}

\begin{thm}\label{necs1} Let $K$ be a nonzero, continuous, isotropic and positive definite kernel on $S^1$.\ In order that it be strictly positive definite it is necessary that, in the representation (\ref{PDM}) for $K_r^1$, the set $\{k: a_{|k|}^{-1/2,-1/2}>0\}$ intersects every full arithmetic progression in $\mathbb{Z}$.
\end{thm}
\begin{proof} We will assume that $\{k: a_{|k|}^{-1/2,-1/2}>0\}\cap (n\mathbb{Z}+j)=\emptyset$ for some $n\geq 2$ and $j \in \{0,1,\ldots, n-1\}$ and will show that $K$ is not strictly positive definite, with the help of Lemma \ref{matrix1}.\
If $a_{|k|}^{-1/2,-1/2}>0$, our assumption implies that $\exp[i2\pi(\pm k-j)/n] \neq 1$.\
Defining $c_\mu=\exp(-i2\pi \mu j/n)$, $\mu=1,2,\ldots,n$, we now have that
\begin{equation*}
\sum_{\mu=1}^n c_\mu e^{i2\pi \mu k/n}=\sum_{\mu=1}^n [e^{i2\pi \mu/n}]^{k-j}=e^{i 2\pi (k-j)/n} \frac{e^{i2\pi (k-j)}-1}{e^{i2\pi (k-j)/n}-1}= 0,
\end{equation*}
and, likewise, $\sum_{\mu=1}^n c_\mu e^{-i2\pi \mu k/n}=0$.\ Thus,
\begin{equation*}
\sum_{\mu=1}^n (\mbox{Re\,}c_\mu) e^{i2\pi \mu k/n}=0, \quad k\in \{k: a_{|k|}^{-1/2,-1/2}>0\}.
\end{equation*}
Since $\mbox{Re\,}c_n \neq 0$, $K$ is not strictly positive definite on $S^1$.
\end{proof}

The proof of the sufficiency of the condition presented above demands arguments from analytic number theory.

\begin{thm} Let $K$ be a real, continuous, isotropic and positive definite kernel on $S^1$.\ In order that it be strictly positive definite is sufficient that, in the representation (\ref{PDM}) for $K_r^1$, the set $\{k: a_{|k|}^{-1/2,-1/2}>0\}$ intersects every full arithmetic progression in $\mathbb{Z}$.
\end{thm}
\begin{proof} Fix $n\geq 2$, distinct points $\theta_1, \theta_2, \ldots, \theta_n$ in $[0,2\pi)$ and $c_1, c_2, \ldots, c_n$ real numbers, not all zero.\ We intend to show that if $\{k: a_{|k|}^{-1/2,-1/2}>0\}$ intersects every full arithmetic progression in $\mathbb{Z}$, then $\sum_{\mu=1}^n c_\mu e^{i\theta_\mu k}\neq 0$ for at least one $k \in \{k : a_k^{-1/2,-1/2}>0\}$.\ In order to do that, we define
\begin{equation*}
b_k=\sum_{\mu=1}^n c_\mu e^{i\theta_\mu k}, \quad k\in \mathbb{Z},
\end{equation*}
and consider the set $\{k: b_k=0\}$.\ Since this set is a linear recurrence, the classical Skolem-Mahler-Lech theorem (\cite[p.25]{everest}) asserts that $\{k : b_k=0\}$ coincides with the union of a finite set $F$ and a finite number of full arithmetic progressions, say, $d_\nu\mathbb{Z}+j_\nu$, $\nu=1,2,\ldots, m$.\ Since the $c_\mu$ are not all zero, $\{k:b_k=0\}\neq \mathbb{Z}$.\ This information allows us to conclude that $\{k: b_k=0\}$ has an empty intersection with at least one full arithmetic progression in $\mathbb{Z}$.\ Indeed, choose $l\in \mathbb{Z}\setminus \{k:b_k=0\}$ and let $p$ be the least common multiple of all the $d_\nu$.\ If for a fixed $\nu$, $p m+l=d_\nu \overline{m} +j_\nu$ for two integers $m$ and $\overline{m}$, then $l=(\overline{m}-pm/d_\nu)d_\nu+j_\nu \in d_\nu\mathbb{Z}+j_\nu$, a contradiction.\ Thus, if the finite set in the union decomposition of $\{k:b_k=0\}$ is empty, the arithmetic progression $p\mathbb{Z}+l$ satisfies what is required.\ Otherwise, we can pick a subset of $p\mathbb{Z}+l$ that is an arithmetic progression itself and avoids the set $F$.\ It is now clear that if $\{k: a_{|k|}^{-1/2,-1/2}>0\}$ intersects every full arithmetic progression in $\mathbb{Z}$, then it cannot be a subset of $\{k: b_k=0\}$.\ Therefore, there exists at least one $k\in \mathbb{Z}$ for which $a_{|k|}^{-1/2,-1/2}>0$ and $\sum_{\mu=1}^n c_\mu e^{i\theta_\mu k}\neq 0$.\ Since the $c_\mu$ are real, then $\sum_{\mu=1}^n c_\mu e^{i\theta_\mu k}\neq 0$ for at least one $k \in \{k: a_k^{-1/2,-1/2}>0\}$.\ Lemma \ref{matrix} implies that $K$ is strictly positive definite on $S^1$.
\end{proof}

Before stating the main result in this section, we will prove the following lemma.

\begin{lem} Let $K$ be a subset of $\mathbb{Z}$.\ If it intersects every arithmetic progression in $\mathbb{Z}$ then it intersects
every arithmetic progression in $\mathbb{Z}$ infinitely many times.
\end{lem}
\begin{proof} Assume $K$ intersects every arithmetic progression in $\mathbb{Z}$ and suppose there exist one, say, $n\mathbb{Z}+j$ which intersects $K$ only a finite number of times.\ Let
\begin{equation*}
np_1+j < np_2+j< \cdots < np_\alpha+j
\end{equation*}
be the elements in the intersection and define $p:=\max\{|p_1|,\ldots,|p_\alpha|\}$.\ Since $p>0$ and
\begin{equation*}
0 \leq 2np+j\leq 2np+n-1 = (2p+1)n-1,
\end{equation*}
 the set $2n(2p+1)\mathbb{Z}+2np+j$ is a full arithmetic progression.\ In addition, we have the inclusion $2n(2p+1)\mathbb{Z}+2np+j \subset n\mathbb{Z}+j$ while $K\cap
(2n(2p+1)\mathbb{Z}+2np+j)=\emptyset$ by our choice of $p$.\ This is a contradiction.
\end{proof}

The main theorem in the appendix is this.

\begin{thm}
Let $K$ be a continuous, isotropic and positive definite kernel on $S^1$ and consider the representation (\ref{PDM}) for $K_r^1$.\ The following assertions are equivalent:\\
$(i)$ $K$ is strictly positive definite on $S^1$;\\
$(ii)$ The set $\{k: a_{|k|}^{-1/2,-1/2}>0\}$ intersects every full arithmetic progression in $\mathbb{Z}$;\\
$(iii)$ The set $\{k: a_{|k|}^{-1/2,-1/2}>0\}$ intersects every full arithmetic progression in $\mathbb{Z}$ infinitely many times.
\end{thm}

%\ack % or \acks
% Put acknowledgements here

\section*{Acknowledgement} The first author was partially supported by CAPES and CNPq, under grant 141908/2015-7.\ The second one by FAPESP, under grant 2014/00277-5.
%%%%%%%%%%%%%%%%%%%%%%%%%%%

\vspace*{2cm}

\noindent V. S. Barbosa and V. A. Menegatto \\
Departamento de Matem\'atica,\\
ICMC-USP - S\~ao Carlos, Caixa Postal 668,\\
13560-970 S\~ao Carlos SP, Brasil\\ e-mails: victorrsb@gmail.com; menegatt@icmc.usp.br

\end{document}